\documentclass[a4paper,oneside,reqno,12pt]{amsart}
\usepackage{graphicx}
\usepackage{graphics}
\usepackage{amssymb, amsmath, amsfonts}
\usepackage{comment}
 \newtheorem{theorem}{Theorem}[section]
 \newtheorem{corollary}[theorem]{Corollary}
 \newtheorem{lemma}[theorem]{Lemma}
 \newtheorem{proposition}[theorem]{Proposition}
 \theoremstyle{definition}
 \newtheorem{definition}[theorem]{Definition}
 \theoremstyle{remark}
 \newtheorem{remark}[theorem]{Remark}
 \newtheorem{example}[theorem]{Example}
 \numberwithin{equation}{section}

\newcommand{\toto}{\rightrightarrows}
\newcommand{\R}{{\mathbb R}}
\newcommand{\N}{{\mathbb N}}

\newcommand{\To}{\longrightarrow}

\def\1{\^{\i}}
\def\2{\u{a}}
\def\3{\c{s}}
\def\4{\^{a}}
\def\5{\c{t}}
\def\a{\alpha}
\def\b{\beta}
\def\e{\epsilon}

\def\l{\lambda}
\def\<{\langle}
\def\>{\rangle}

\DeclareMathOperator*\cl{cl}
\DeclareMathOperator*\co{co}

\DeclareMathOperator*\inte{int}

\begin{document}
\title[Perturbed Weak Vector  Equilibrium Problems]{On Perturbed Weak Vector  Equilibrium Problems under new Semi-continuities}



\author{Szil\'ard L\' aszl\' o}

\address{Department of Mathematics\\ Technical University of Cluj-Napoca\\
              Str. Memorandumului nr. 28, 400114 Cluj-Napoca, Romania.}
              \email{laszlosziszi@yahoo.com}


\thanks{This work was supported by a grant of Ministry of Research and Innovation, CNCS - UEFISCDI, project number PN-III-P4-ID-PCE-2016-0190, within PNCDI III}
\begin{abstract}
In this paper we introduce a new semicontinuity notion, which is weaker than upper semicontinuity, and assures the closedness of the sets $G(y)=\{x\in K: f(x,y)\not\in -\inte C\}.$ Furhter, this semicontinuity is also closed under addition. These two properties make our new semicontinuity applicable in situations where other semicontinuities, like quasi upper semicontinuity or order upper semicontinuity, fail. The above emphasized properties are some key tools in order to provide new sufficient conditions that ensure the existence of the solution of a perturbed weak vector equilibrium problem in Hausdorff topological vector spaces ordered by a cone. Further, we introduce a dual problem and we provide conditions that assure that every solution of the dual problem is also a solution of the perturbed weak vector equilibrium problem. We apply the results obtained to Ekeland vector variational principles.
\end{abstract}

\maketitle

\section{Introduction}

Let $X$ and  $Z$ be  Hausdorff topological vector spaces, let $K\subseteq X$ be a nonempty set and let $C\subseteq Z$ be a  convex and pointed cone. Assume that the interior of the cone $C$, denoted by $\inte C$, is nonempty and consider the mappings $f,g:K\times K\To Z.$  The perturbed weak vector equilibrium problem  consists in finding $x_0\in K$, such that
\begin{equation}\label{p1}
f(x_0,y)+g(x_0,y)\not\in-\inte C,\,\forall y\in K.
\end{equation}

The dual vector equilibrium problem of (\ref{p1}) is defined as:
 Find $x_0\in K$, such that
\begin{equation}\label{p2}
g(x_0,y)\not\in-\inte C,\,\forall y\in K.
\end{equation}

The study of the problems (\ref{p1}) and (\ref{p2}) is motivated by the following setting.
Assume that the weak vector equilibrium problem, which consists in finding $x_0\in K$ such that $f(x_0,y)\not\in-\inte C$ for all $y\in K$, has no solution though the diagonal condition $f(x,x)=0$, (or more general $f(x,x)\in C$), for all $x\in K$ holds. Then, we may study instead a  perturbed equilibrium problem (see also \cite{durea,FQ}) and provide assumptions on the perturbation function $g$, such that the problem which consists in finding $x_0\in K$ such that $f(x_0,y)+g(x_0,y)\not\in-\inte C$ for all $y\in K,$  has a solution. Moreover, for an appropriate perturbation $g$ the dual problem, that is, find $x_0\in K$ such that $g(x_0,y)\not\in-\inte C$ for all $y\in K$, has a solution. Hence, it is  worthwhile to obtain conditions that assure the the solution set of  (\ref{p2}) is included in the solution set of (\ref{p1}). This setting may have some important consequences. Indeed, by taking $X$ a Banach space and $g(x,y)=\e\|x-y\|e$, where $\e>0$ and $e\in C\setminus\{0\},$ a solution of the perturbed vector equilibrium problem is called $\e$-equilibrium point, see \cite{BKP,BKP1}. Further, special cases of the perturbed vector equilibrium problems lead to some deep results such as Deville-Godefroy-Zizler perturbed equilibrium principle or Ekeland vector variational principle, see \cite{FQ}.

In this paper, we obtain some existence results of the solution for the vector equilibrium problem (\ref{p1}) and (\ref{p2}). Some of our conditions and the techniques used are new in the literature. Several examples and counterexamples circumscribe our research and show that our conditions are essential.
The paper is organized as follows. In the next section, we introduce some preliminary notions and the necessary apparatus that we need in order to obtain our results. In section 3 and section 4 we state our results concerning on perturbed weak vector equilibrium problems. Our conditions, which ensure the solution existence of the above mentioned vector equilibrium problems are  considerably weakening the existing conditions from the literature.  We pay a special attention to the case when the set $K$ is a closed subset of a reflexive Banach  space.  Finally, we apply our results to Ekeland  vector variational principles.
\section{Preliminaries}

Let $X$ be a real Hausdorff topological vector space.  For a non-empty set $D\subseteq X$, we denote by $\inte D$ its interior, by $\cl D$ its closure and by $\co D$ its convex hull.   Recall that a set $C\subseteq X$ is a cone, iff $tc\in C$ for all $c\in C$   and $t\ge 0.$ The cone $C$ is convex if $C+C=C,$ and pointed if $C\cap (-C)=\{0\}.$ Note that a closed, convex and pointed cone $C$ induce a partial ordering on $Z$, that is $z_1\le z_2\Leftrightarrow z_2-z_1\in C.$ In the sequel when we use $\inte C,$  we tacitly assume that the cone $C$ has nonempty interior. Following the same logical approach, one can introduce the strict inequality  $z_1< z_2\Leftrightarrow z_2-z_1\in \inte C,$ or $z_1< z_2\Leftrightarrow z_2-z_1\in  C\setminus\{0\}.$ These relations lead to  saying, that $z_1\not< z_2\Leftrightarrow z_2-z_1\not\in-\inte C$, or $z_1\not< z_2\Leftrightarrow z_2-z_1\not\in-C\setminus\{0\}.$ It is an easy exercise to show that   $\inte C+C= \inte C.$

Let   $Z$ be  another  Hausdorff topological vector space, let $K\subseteq X$ be a nonempty set and let $C\subseteq Z$ be a  convex and pointed cone.

A map $f:K\subseteq X\To Z$ is said to be C-upper semicontinuous at $x\in K$ (\cite {BPT}) iff for any neighborhood $V$ of $f(x)$ there exists a neighborhood $U$ of $x$ such that $f(u)\in  V-C$ for all $u\in U\cap K$. Obviously, if $f$ is continuous at $x\in K,$ then it is  also C-upper semicontinuous at $x\in K$.
Assume that $C$ has nonempty interior.  According to \cite{Ta} $f$ is C-upper semicontinuous at $x\in K,$ if and only if, for any $k\in\inte C$, there exists a neighborhood $U$ of $x$ such that
$$f(u) \in f(x) + k -\inte C\mbox{ for all }u \in U\cap K.$$

The map $f:K\To Z$ is said to be C-lower semicontinuous at $x\in K$ iff the map $-f$ is C-upper semicontinuous at $x\in K.$

\begin{definition} Let $K\subseteq X$ be convex. The function $f : K\to Z$ is called $C$-convex on $K$, iff  for all $x,y\in K$ and $t\in[0,1]$ one has
$$tf(x)+(1-t)f(y)-f(tx+(1-t)y)\in C.$$
\end{definition}

Note that the function $f : K\to Z$ is  $C$-convex, iff  for all $x_1,x_2,\ldots, x_n\in K$, $n\in \N$ and $\lambda_i\ge 0,\, i\in\{1,2,\ldots,n\},$ with
$\sum_{i=1}^n \lambda_i=1,$  one
has
$$ \sum_{i=1}^n \l_i f(x_i)-f\left(\sum_{i=1}^n \l_ix_i\right) \in C.$$
We will  use the following notations for the open, respectively
closed, line segments in $X$ with the endpoints $x$ and $y$
\begin{eqnarray*}
]x,y[ &:=&\big\{z\in X:z=x+t(y-x),\,t\in ]0,1[\big\}, \\
\lbrack x,y] &:=&\big\{z\in X:z=x+t(y-x),\,t\in \lbrack 0,1]\big\}.
\end{eqnarray*}
The line segments $]x,y],$ respectively $[x,y[$ are defined similarly.

\section{On some new and old semicontinuity notions}

In this section we provide a new closedness type condition, unknown till now in the literature, that is essential in obtaining  the existence of a solution for the problem (\ref{p1}) and (\ref{p2}), respectively. This condition is based on some sequential properties of a bifunction and leads to a notion that we call a-upper semicontinuity. We show that our closedness type condition is weaker than those used in the literature, which results from the fact that a-usc is weaker than C-usc. Further we compare our new semicontinuity notion with those used in the literature. We emphasize here two properties of a-usc: it ensures the closedness of the "upper"-level sets relative to the "ordering" $\not\ge,$ which makes a-usc a key tool in study the solution existence of weak vector equilibrium problem, and is closed under addition which makes a-usc suitable for studying perturbed weak vector equilibrium problems. We also underline that the upper semicontinuities used in the literature, excepting C-usc, fail to have the properties mentioned above.


\begin{remark}\rm
In order to obtain solution existence for the problem (\ref{p2}) we need conditions that assure for every $y\in K$ the closedness of the sets $G(y)=\{x\in K: g(x,y)\not\in-\inte C\},$  that is, for every  net $(x_\a)\subseteq G(y),\, \lim x_\a=x$ one has: $g(x,y)\not\in-\inte C.$ This is a key tool in order to obtain solution existence of problem (\ref{p2}), and in the literature usually is obtained via the  C-upper semicontinuity of the mapping $x\longrightarrow g(x,y).$ Next we provide a new, much weaker condition.
\end{remark}
\begin{lemma}\label{primclosed} Let $X$ be a Hausdorff space and let $Z$ be a Hausdorff  topological vector space,  let $C\subseteq Z$ be a convex and pointed cone with nonempty interior and let $K$ be a nonempty and closed subset of $X$. Let $y\in K$ and consider the mapping $g :K \times K \longrightarrow{Z}.$ Assume that one of the following conditions hold.
\begin{itemize}
\item[(a)] The mapping $x\longrightarrow g(x,y)$ is C-upper semicontinuous on $K.$
\item[(b)] For every  $x\in K$ and for every  net $(x_\a)\subseteq K,\, \lim x_\a=x$ there exists a net $z_\a\subseteq Z,\, \lim z_\a=z$ such that $g(x_\a,y)-z_\a\in -C$ and $g(x,y)-z\in C.$
\end{itemize}
Then, the set $G_g(y)=\{x\in K: g(x,y)\not\in-\inte C\}$ is closed.
\end{lemma}
\begin{proof}  Let us prove (a). Consider the net $(x_\a)\subseteq G(y)$ and let $\lim x_\a=x_0.$ Assume that $x_0\not\in G(y).$ Then $g(x_0,y)\in-\inte C$.
 According to the assumption the function $x\longrightarrow g(x,y)$ is C-upper semicontinuous at $x_0$, hence  for every $k\in\inte C$ there exists $U,$ a neighborhood of $x_0,$ such that $g(x,y)\in g(x_0,y)+k-\inte C$ for all $x\in U\cap K.$  But then, for $k=-g(x_0,y)\in\inte C$,  one obtains that there exits $\a_0$ such that $g(x_\a,y)\in -\inte C,$ for $\a\ge\a_0,$ which contradicts the fact that $(x_\a)\subseteq G(y_0)$.
Hence $G(y)\subseteq K$ is closed.

For (b) consider the net $(x_\a)\subseteq G(y)$ and let $\lim x_\a=x_0.$ Assume that $x_0\not\in G(y).$ Then $g(x_0,y)\in-\inte C$. But by the assumption there exists a net $z_\a\subseteq K,\, \lim z_\a=z$ such that $g(x_\a,y)-z_\a\in -C$ and $g(x_0,y)-z\in C.$ From the latter relation we get $z\in-\inte C$, and since $-\inte C$ is open we have that there exists $\a_0$ such that  $z_\a\in -\inte C$ for every $\a\ge \a_0.$ But then, $g(x_\a,y)\in z_\a -C$ and $\inte C+C=\inte C$ leads to $g(x_\a,y)\in -\inte C$ for $\a\ge\a_0,$ contradiction.
\end{proof}

\begin{remark}\label{ab}\rm  Note that condition (b) is new in the literature, however some similar results are obtained in \cite{larx}. In what follows we show that (a) implies (b). Further, by an example we show that (b) is indeed weaker than (a).
\end{remark}

\begin{proposition}  Let $X$ be a Hausdorff space and let $Z$ be a Hausdorff  topological vector space,  let $C\subseteq Z$ be a convex and pointed cone with nonempty interior and let $K$ be a nonempty and closed subset of $X$. Let $y\in K$ and consider the mapping $g :K \times K \longrightarrow{Z}.$ Assume that the mapping $x\longrightarrow g(x,y)$ is C-upper semicontinuous on $K.$ Then, for every  $x\in K$ and for every  net $(x_\a)\subseteq K,\, \lim x_\a=x$ there exists a net $z_\a\subseteq Z,\, \lim z_\a=z$ such that $g(x_\a,y)-z_\a\in -C$ and $g(x,y)-z\in C.$
\end{proposition} \begin{proof}
    Let $x_0\in K$ and consider the  net $(x_\a)\subseteq K,\, \lim x_\a=x_0.$ We show that there exists a net $z_\a\subseteq Z,\, \lim z_\a=z$ such that $g(x_\a,y)-z_\a\in -C$ and $g(x_0,y)-z\in C.$

    We have that for every neighbourhood of $g(x_0,y)$, say $V$, there exists $U,$ a neighbourhood of $x_0$, such that $g(x,y)\in V-C$ for all $x\in U.$ Obviously on can take $V$ closed, hence there exists a net $(s_\a)\subseteq V$ such that $\lim s_\a=g(x_0,y).$
 Since $\lim x_\a=x_0$  we have that $x_\a\in U$ for every $\a\ge\a_0,$ hence $g(x_\a,y)\in V-C$ for every $\a\ge\a_0.$ This leads to
 $g(x_\a,y)-s_\a\in -C$ for every $\a\ge\a_0.$ Hence one can take $z_\a=g(x_\a,y)$ if $\a<\a_0$ and $z_\a=s_\a$ for $\a\ge\a_0$ and the conclusion follows.
 \end{proof}

In what follows we provide an example to emphasize that the condition (b) in Lemma \ref{primclosed} is in general weaker than condition (a).

\begin{example}\rm\label{ex3.2} Let $C=\{(x_1,x_2)\in \R^2: x_1^2\le x_2^2,\,x_2\ge 0\}.$ Obviously $C$ is a closed convex and pointed cone in $\R^2$ with nonempty interior. Consider the bifunction $$g:[0,1]\times [0,1]\To \R^2,\,g(x,y)=\left\{\begin{array}{ll}(x+y,x)\mbox{ if } 0\le x\le\frac12,\\(2x+y,0)\mbox{ if }\frac12<x\le 1.\end{array} \right.$$
Then, for every fixed $y\in [0,1]$ the mapping $x\To g(x,y)$ is continuous, hence it is also C-upper semicontinuous, at every $x\in[0,1]\setminus\{\frac12\}.$
We show that $x\To g(x,y)$ is not C-upper semicontinuous at the point $x=\frac12$ for every fixed $y\in[0,1].$ For this it is enough to show that for all $\e>0$ there exists $(c_1,c_2)\in\inte C$ and $x\in \left]\frac12-\e,\frac12+\e\right[$ such that $(c_1,c_2)+g(\frac12,y)-g(x,y)\not\in\inte C.$ Hence, for fixed $\e>0$ ($\e<2$) let $(c_1,c_2)=(\e-1,1)\in\inte C.$ Consider $x=\frac12+\frac{\e}{2}\in \left]\frac12-\e,\frac12+\e\right[.$ Then, $(c_1,c_2)+g(\frac12,y)-g(x,y)=\left(-\frac32,\frac32\right)\not\in\inte C.$

Next, we show that condition (b) in Lemma \ref{primclosed} holds for $x=\frac12$ and every fixed $y\in[0,1].$ Obviously instead of nets one can consider sequences, hence let $(x_n)\subseteq [0,1],\,x_n\To \frac12,\,n\To\infty.$ We must show, that there exists a sequence $(z_n)\subseteq \R^2,\, \lim z_n=z$ such that $g(x_n,y)-z_n\in -C$ and $g\left(\frac12,y,\right)-z\in C.$

Let $z_n=(x_n+y,x_n).$ Then $g(x_n,y)-z_n=(0,0)\in-C$ for every $n\in\N,$ such that $x_n\le\frac12$ and $g(x_n,y)-z_n=(x_n,-x_n)\in-C$ for every $n\in\N,$ such that $x_n>\frac12.$ Obviously $\lim z_n=z=\left(\frac12+y,\frac12\right)$, hence $g\left(\frac12,y\right)-z=(0,0)\in C.$
\end{example}

Assumption (b) in Lemma \ref{primclosed} leads to the following definition.

\begin{definition} Let $X$ be a Hausdorff space and let $Z$ be a Hausdorff  topological vector space,  let $C\subseteq Z$ be a convex and pointed cone with nonempty interior and let $K$ be a nonempty subset of $X$. One says that the mapping $g :K\longrightarrow{Z}$ is almost upper semicontinuous (briefly a-usc) at $x_0\in K$ if for every  net $(x_\a)\subseteq K,\, \lim x_\a=x_0\in K$ there exists a net $z_\a\subseteq Z,\, \lim z_\a=z$ such that $g(x_\a)-z_\a\in -C$ and $g(x_0)-z\in C.$  Similarly, the mapping $g :K\longrightarrow{Z}$ is almost lower semicontinuous (briefly a-lsc) at $x_0\in K$ if $-g$ is a-usc at $x_0.$ One says that $g$ is a-usc (a-lsc) on $K$ if $g$ is a-usc (a-lsc) at every $x\in K.$
\end{definition}

\begin{remark} Note that Example \ref{ex3.2} shows that the a-usc property is weaker than usc. In the literature there are several semicontnuity concepts weaker than usc, the most prominent of them are the quasi upper semicontinuity (q-usc) and the order upper semicontinuity (o-usc) (see \cite{FQ1}.  We show next that a-usc implies q-usc 
and in the real valued case, i.e. $Z=\R$ and $C=\R_+$, a-usc, q-usc and o-usc collapse into the well known upper semicontinuity concept of real valued functions.
\end{remark}

In our setting these concepts can be defined as follows.

The map $g:K\To Z$ is said to be quasi upper semicontinuous (q-usc) at $x_0 \in K$ )(see \cite{Th}) iff,
for each $k\in Z$ such that $k+g(x_0)\not\in C$, there exists a neighborhood
$U$ of $x_0$ such that $k+g(x)\not\in C$ for each $x \in U\cap K.$

The map $g:K\To Z$ is said to be order upper semicontinuous (o-usc) at $x_0 \in K$ (see \cite{FQ1}) iff, for each net  $(x_\a) \subseteq K,\, \lim x_\a=x_0$ for which there exists a net $(z_\a) \subseteq Z,\, \lim z_\a=0$ such that the net $g(x_\a) + z_\a$ is non-decreasing, i.e. $(g(x_\b)+z_\b)-(g(x_\a) + z_\a)\in C,$ for all $\b>\a$, there exists a net $(w_\a) \subseteq Z,\,\lim w_\a=0$ such that $g(x_0)-g(x_\a)+w_\a\in C$ for all $\a$.

Proposition 7 \cite{FQ1} shows that in our setting, that is, the interior of $C$ is nonempty, q-usc implies o-usc (at least in  Banach  spaces). Hence, so far we know that in our setting  $g$ usc at $x_0\implies$ $g$ a-usc at $x_0$, and $g$ usc at $x_0\implies$ $g$ q-usc at $x_0\implies$ $g$ o-usc at $x_0$. We show next that the a-usc implies q-usc in a slightly weaker setting than the framework of \cite{FQ1}, that is, $X$ is a metric space and $C$ is a closed, convex pointed cone with nonempty interior. Further by  an example we show that q-usc and o-usc are indeed weaker than a-usc.

\begin{proposition} Let $(X,d)$ be a metric space and let $Z$ be a Hausdorff  topological vector space,  let $C\subseteq Z$ be a closed, convex and pointed cone with nonempty interior. Consider the  mapping $g:X\To Z$ and assume that $g$ is a-usc at $x_0\in X.$ Then, $g$ is q-usc at $x_0.$
\end{proposition}
\begin{proof} Assume that $g$ is not q-usc at $x_0$. Then, there exists $k\in Z$ such that $k+g(x_0)\not\in C$ and for every neighbourhood of $x_0$ say $U$, there exists $x\in U$ such that $k+g(x)\in C.$ We show that there exists a sequence $(x_n)\subseteq X,\, x_n\To x_0$ such that $k+g(x_n)\in C$ for all $n\in \N.$ Indeed, let $r_1=1$ and consider $U_1=B(x_0,r_1)$ a neighbourhood of $x_0.$ Then there exists $x_1\in U_1$ such that $k+g(x_1)\in C.$ Let $r_2<\min\left(d(x_1,x_0),\frac12\right)$ and consider $U_2=B(x_0,r_2).$ Then there exists $x_2\in U_2$ such that $k+g(x_2)\in C.$ By continuing the procedure we obtain the sequence $x_1, x_2,...,x_n\in X$ and the real sequence $r_1>r_2>...>r_n,$ such that $r_i<d(x_i,x_0)<r_{i-1}\le \frac{1}{i-1},\, i=1,2,...,n$.
Since $r_n\le \frac1n$ one has $r_n\To 0,\,n\To\infty$ which shows that $x_n\To x_0,\,n\To\infty.$

Now by the a-usc property of $g$ in $x_0,$ we have that there exists a sequence $(z_n)\subseteq Z,\, z_n\To z$ such that
 $$z_n\in g(x_n)+C\mbox{ and } g(x_0)\in z+C,\,\forall n\in\N.$$
Hence, $k+z_n\in k+g(x_n)+C\subseteq C+C=C$ and since $C$ is closed, by taking the limit $n\To\infty$ we have $k+z\in C.$

On the other hand, $k+g(x_0)\in k+z+C\subseteq C+C=C$ contradiction.

Hence, $g$ is q-usc at $x_0$ and the proof is complete.
\end{proof}
The next example is inspired by \cite{FQ1} and shows that q-usc and also o-usc is indeed weaker than a-usc.

\begin{example}\label{ex3.3} Let $g: \R \To \R^2$ defined by
$ g(x) =\left\{\begin{array}{ll} \left(1, -\frac{1}{|x|}\right)\mbox{ if }x \neq 0,\\
(0, 0)\mbox{ if }x = 0.
\end{array}\right.$
Consider $C=\R^2_+$ the nonnegative orthant of $\R^2$ which is obviously a closed convex and pointed cone with nonempty interior.

According to \cite{FQ1}, $g$ is q-usc (and  also o-usc) at $x_0=0$ but not usc at $0.$ We show that $g$ is not a-usc at $x_0=0$ either. Obviously instead of nets one can consider sequences. Indeed, let $x_n\To 0$ and assume that there exists $z_n=(u_n,v_n)\in \R^2,\, z_n\To (u,v)=z\in \R^2$ such that
$g(x_n)-z_n\in-\R^2_+$ and $g(0)-z\in\R^2_+.$ Hence, we have  $1-u_n\le 0$ and $u\le 0$ where $u=\lim_{n\To\infty} u_n$. But this is impossible, which shows that $g$ is not a-usc at $x_0=0$.
\end{example}

Next we introduce another semicontinuity concept which is weaker than a-usc and incomparable with q-usc or o-usc.

\begin{definition} Let $X$ be a Hausdorff space and let $Z$ be a Hausdorff  topological vector space,  let $C\subseteq Z$ be a convex and pointed cone with nonempty interior and let $K$ be a nonempty subset of $X$. One says that the mapping $g :K\longrightarrow{Z}$ is w-upper semicontinuous (briefly w-usc) at $x_0\in K$ if there exists the  nets $(x_\a)\subseteq K,\, \lim x_\a=x_0\in K$ and $z_\a\subseteq Z,\, \lim z_\a=z$ such that $g(x_\a)-z_\a\in -C$ and $g(x_0)-z\in C.$  Similarly, the mapping $g :K\longrightarrow{Z}$ is w-lower semicontinuous (briefly w-lsc) at $x_0\in K$ if $-g$ is w-usc at $x_0.$ One says that $g$ is w-usc (w-lsc) on $K$ if $g$ is w-usc (w-lsc) at every $x\in K.$
\end{definition}

It is obvious that every a-usc map is also w-usc. We have seen by our previous analysis that a-usc implies q-usc. Nevertheless, the w-usc and q-usc properties are incomparable in general, meanwhile in case $Z=\R$ w-upper semicontinuity is the weakest among the above presented notions (which all collapse into upper semicontinuity, meanwhile w-usc is a much weaker property). Indeed, Example \ref{ex3.3} also shows that q-usc or o-usc does not imply w-usc. Next, inspired by Example \ref{ex3.2} we provide  a w-upper semicontinuous map which is not q-upper semicontinuous.

 \begin{example}\label{ex3.4} Let $g: \R \To \R^2$ defined by
$ g(x) =\left\{\begin{array}{ll} (x,2x)\mbox{ if }x\le \frac12,\\ 
(2x, 2x)\mbox{ if }x>\frac12.
\end{array}\right.$
Consider $C=\{(x_1,x_2)\in \R^2: x_1^2\le x_2^2,\,x_2\ge 0\}.$ Obviously $C$ is a closed convex and pointed cone in $\R^2$ with nonempty interior.

For $k=(-1,-1)$ we have $k+g\left(\frac12\right)=\left(-\frac12,0\right)\not\in C.$ On the other hand, for $x\in\left]\frac12,1\right]$ one has
$$k+g(x)=(2x-1,2x-1)\in C.$$
This shows that in any neighbourhood $U$ of $\frac12$ there exists points $x\in U$ (actually for every $x\in U,\,x>\frac12$) for which $k+g(x)\in C.$ Consequently $g$ is not q-usc  at $\frac12.$ Moreover, since $\inte C\neq \emptyset,$ $g$ is not o-usc either at $\frac12.$

We show that $g$ is w-usc at $\frac12.$ Let $x_n=\frac{n}{2n+1}\To\frac12$ and $z_n=\left(\frac{n}{2n+1},\frac{2n}{2n+1}\right)\To \left(\frac12,1\right)=z.$ Then, $g(x_n)-z_n=(0,0)\in -C$ and $g\left(\frac12\right)-z=(0,0)\in C.$ Hence, $g$ is w-usc at $x_0=\frac12.$ Since $g$ is continuous on $\R\setminus\{x_0\}$ we have that $g$ is actually w-usc on $\R.$
\end{example}
\begin{remark}
Let now $Z=\R$ and $C=\R_+$. According to Corollary 5 \cite{FQ1} in this case the concepts of usc, q-usc and o-usc coincide. We show that a-usc also collapse in this case into the usual usc property. Indeed, let $g:X\To\R$ be a function and assume  that $g$ is a-usc at $x_0\in X.$ Then, for every net $(x_\a)\subseteq X,\,x_\a\To x_0$ there exists a net $(z_\a)\subseteq \R,\, z_\a\To z\in \R$ such that
$$g(x_a)\le z_\a\mbox{ and }g(x_0)\ge z\mbox{ for all }\a.$$
But this shows that $g(x_0)\ge z=\limsup z_\a\ge\limsup g(x_\a)$, for every net $x_\a\To x_0$, in other words, $g$ is usc at $x_0.$ We cannot use the same arguments if $g$ is w-usc only, since in this case we assume the existence of a single net $(x_\a)\subseteq X$ with the properties emphasized above. However, we can say that in case of w-usc at $x_0$ one has $g(x_0)\ge \liminf_{x_\a\To x_0} g(x),$ hence this notion may also contain some of the lower semicontinuous functions. Even more, the next example shows that in real-valued case w-usc is weaker than usc.
\end{remark}
\begin{example} Let $g:\R\To\R$ defined by
$ g(x) =\left\{\begin{array}{lll} x,\mbox{ if }x < 0,\\
\frac12,\mbox{ if }x=0,\\
x+1,\mbox{ if }x>0.
\end{array}\right.$
Then, obviously $g$ is not usc at $x_0=0$, but for the sequence $(x_n)\subseteq \R,\,x_n=-\frac1n\To 0$ there exists the sequence $(z_n)\subseteq \R,\,z_n=\frac{1}{n}\To 0=z$, such that $f(x_n)\le z_n$ for all $n\in\N$ and $\frac12=f(x_0)\ge z=0.$
\end{example}

\begin{remark}
One may ask what is the reason for introducing another semicontinuity notion. The utility of the concept a-usc is emphasized by the fact the sum of two maps with this property has the a-usc property. Indeed, let $f,g:K\To Z$ be a-usc at $x_0\in K.$ Consider the net $(x_\a)\subseteq K,\,x_\a\To z_0.$ Then there exist the nets $(z_\a^1),(z_\a^2)\subseteq Z,\, z_\a^1\To z^1,\,z_\a^2\To z^2$ such that $f(x_\a)-z_\a^1,\,g(x_\a)-z_\a^2\in -C$ and $f(x_0)-z^1,\,g(x_0)-z^2\in C.$ Since $C$ is convex one has $(f+g)(x_\a)-(z_\a^1+z_\a^2)\in -C$ and $(f+g)(x_0)-(z^1+z^2)\in C.$ 
The fact that the a-usc  property is closed under addition is very useful when we deal with perturbed problems, where the terms in the sum are usually evaluated separately.  In general, neither the sum of two q-usc, nor the sum of two o-usc maps does not remain q-usc or o-usc as is shown in \cite{FQ1}. Moreover, it can easily be realized that the sum of two w-usc maps need not to be w-usc. Nevertheless, w-usc may also become  useful in the study  of perturbed equilibrium problems as will be emphasized in the next section.

 Another utility of the concept of a-usc is that, according to Lemma \ref{primclosed}, ensures the closedness of the set $G(y)=\{x\in K: g(x,y)\not\in-\inte C\}.$
None of the q-usc, o-usc or w-usc properties can be used in this setting as  the  next examples  show.
\end{remark}
\begin{example} Consider the mapping $g: [-1,1]\times [-1,1] \To \R^2$ defined by
$$ g(x,y) =\left\{\begin{array}{ll} \left(x+y,y -\frac{1}{|x|}\right),\mbox{ if }x \neq 0,\\
(-1+y, -1+y),\mbox{ if }x = 0.
\end{array}\right.$$
Consider $C=\R^2_+$ the nonnegative orthant of $\R^2$ which is obviously a closed convex and pointed cone with nonempty interior.
Then, the map $x\To g(x,0)$ is q-usc (and also o-usc, but not a-usc or w-usc) at $x_0=0$ and the set $G(0)=\{x\in[-1,1]:g(x,0)\not\in-\inte C\}$ is not closed.

Indeed,  $x\To g(x,0)$ is q-usc at $x_0=0$, since for every $k=(k_1,k_2)\in\R^2$ such that $k+g(x_0,0)\not\in C$ there exists $\e>0$ such that $k_2-\frac{1}{|x|}<0$ for all $x\in]-\e,\e[\setminus\{0\}.$ If $k_2\le 0$ one can take $\e>0$ arbitrary, if $k_2>0$ then one can take $\e=\frac{1}{k_2}.$ Consequently, for every $x\in]-\e,\e[$ one has that $k+g(x_0,0)\not\in C$ which shows that
$x\To g(x,0)$ is q-usc at $x_0=0$.

Further, $x\To g(x,0)$ is not a-usc at $x_0=0$. Indeed, assume that for the sequence $(x_n)\subseteq [-1,1],\, x_n\To 0$ there exists the sequence $(z_n)\subseteq \R^2,\,z_n=(u_n,v_n)\To (u,v)=z\in\R^2$ such that $f(x_n,0)-z_n\in-\R^2_+$ and $g(0,0)-z\in\R^2_+.$ This leads to $x_n-u_n\le 0$ and $-1-u\ge 0,$ impossible. Since we considered an arbitrary sequence $(x_n)\subseteq [-1,1],\, x_n\To 0$, actually we have shown that $g$ is not w-usc at $x_0=0.$

On the other hand  $G(0)=]0,1]$ which obviously  is  not closed.
\end{example}

\begin{example} Let $g: \R\times \R \To \R^2$ defined by
$ g(x,y) =\left\{\begin{array}{lll} (-1,y-1)\mbox{ if }x\le -1,\\
(0, y-1)\mbox{ if }x\in\left]-1, 0\right],\\
(1,y-1)\mbox{ if }x>0.
\end{array}\right.$
Consider $C=\R^2_+.$ Obviously $C$ is a closed convex and pointed cone in $\R^2$ with nonempty interior.

It is easy to check that the map $x\To g(x,0)$ is w-usc on $\R.$  On the other hand,
$$G(0)=\{x\in \R:g(x,0)\not\in-\inte C\}=(-1,\infty)$$ is not closed.
\end{example}

\section{Coincidence of solutions}
In this section we give some natural conditions that assure the solution set of \eqref{p2} is included in the solution set of \eqref{p1}. Hence, we can deduce the solution existence of the perturbed weak vector equilibrium problem from the nonemptyness of the solution set of the dual problem. We also show that our conditions are essential, more precisely if we drop one of them, then there might exist solutions of \eqref{p2} which are not solutions of \eqref{p1}.

We start our analysis with the following straightforward result.

\begin{lemma}\label{notin} Let $Z$ be a Hausdorff  topological vector space,  let $C\subseteq Z$ be a convex and pointed cone with nonempty interior. Then, $a\in Z,\,a\not\in-\inte C$ iff there exists $b\in Z,\,b\not\in-\inte C$ such that $a-b\in C.$
\end{lemma}
\begin{proof} Assume that $a\not\in-\inte C$ and let $b=a.$ Then $b\not\in-\inte C$ and $a-b=0\in C.$

Assume now that  $b\not\in-\inte C$ and $a-b\in C.$ Then $b\in a-C$ and by using the fact that $-\inte C-C=-\inte C$ we get $a\not\in-\inte C$.
\end{proof}

  Let $X$ be a Hausdorff space and let $Z$ be a Hausdorff  topological vector space,  let $C\subseteq Z$ be a convex and pointed cone with nonempty interior and let $K$ be a nonempty subset of $X$. Consider the mappings $f,g :K \times K \longrightarrow{Z}$. Next we provide conditions that assure the solution existence of the perturbed weak vector equilibrium problem.\\

(A$_1$) There exists $x_0\in K$ such that for every fixed $y\in K,\,y\neq x_0$  there exists the nets $(x_\a)\subseteq K,\, \lim x_\a=x_0$ and $(w_\a)\subseteq Z\setminus-\inte C$  such that
$w_\a-f(x_\a,y)-g(x_0,y)\in- C,$ for all $\a$,
and the map $x\To f(x,y)$ is a-upper semicontinuous at $x_0.$\\

(A$_2$) There exists $x_0\in K$ such that for every fixed $y\in K,\,y\neq x_0$  there exists the nets $(y_\a)\subseteq K,\, \lim x_\a=y$ and $(w_\a)\subseteq Z\setminus-\inte C$  such that $w_\a-f(x_0,y_\a)-g(x_0,y)\in- C,$ for all $\a,$ and the map $z\To f(x_0,z)$ is a-upper semicontinuous at $y.$\\

(A$_3$) There exists $x_0\in K$ such that for every fixed $y\in K,\,y\neq x_0$  there exists the nets $(x_\a),(u_\a)\subseteq K,\, \lim x_\a=x_0,\,\lim u_\a=x_0$ and $(w_\a)\subseteq Z\setminus-\inte C$  such that $w_\a-f(x_\a,y)-g(u_\a,y)\in-  C,$ for all $\a,$ and the maps $x\To f(x,y)$  and $x\To g(x,y)$ are  a-upper semicontinuous at $x_0.$\\

(A$_4$) There exists $x_0\in K$ such that for every fixed $y\in K,\,y\neq x_0$ there exist the nets $(y_\a),(v_\a)\subseteq K,$ $\lim y_\a=y,$ $\lim v_\a=y$ and $(w_\a)\subseteq Z\setminus-\inte C$  such that $w_\a-f(x_0,y_\a)-g(x_0,v_\a)\in- C,$ for all $\a,$ and the maps $z\To f(x_0,z)$ and $z\To g(x_0,z)$ are a-upper semicontinuous at $y.$\\

(A$_5$) There exists $x_0\in K$ such that for every fixed $y\in K,\,y\neq x_0$ there exist the nets $(x_\a),(y_\a)\subseteq K,\,\lim x_\a=x_0,\, \lim y_\a=y$  and $(w_\a)\subseteq Z\setminus-\inte C$ such that
$w_\a-f(x_\a,y)-g(x_0,y_\a)\in- C,$ for all $\a$, and the map $x\To f(x,y)$ is a-upper semicontinuous at $x_0$, the map $z\To g(x_0,z)$ is a-upper semicontinuous at $y.$\\
\begin{remark} Note that the role of $f$ and $g$ in (A$_1$),(A$_2$)and (A$_5$) can be interchanged.
A main result of  this section is the following.
\end{remark}
\begin{theorem}\label{principal}   Let $X$ be a Hausdorff space and let $Z$ be a Hausdorff  topological vector space,  let $C\subseteq Z$ be a convex and pointed cone with nonempty interior and let $K$ be a nonempty subset of $X$. Consider the mappings $f,g :K \times K \longrightarrow{Z}$.
Assume that the one of the conditions (A$_1$)-(A$_5$) holds.

Then, $f(x_0,y)+g(x_0,y)\not\in-\inte C$ for all $y\in K,\, y\neq x_0.$
\end{theorem}
\begin{proof} We prove for (A$_5$) only, the other cases can be proved similarly. Let $y\in K,\,y\neq x_0$ fixed. Since the map $x\To f(x,y)$ is a-upper semicontinuous at $x_0$ there exists a net $u_\a\subseteq Z,\,u_\a\To u$, such that
$$f(x_\a,y)\in-C+u_\a\mbox{ and } u\in -C+f(x_0,y).$$

Hence, by assumption $w_\a\not\in-\inte C$ for all $\a$ and $$w_\a\in- C+f(x_\a,y)+g(x_0,y_\a)\subseteq -C -C+u_\a+g(x_0,y_\a)=- C+u_\a+g(x_0,y_\a),$$  which shows that $u_\a+g(x_0,y_\a)\not\in-\inte C$ for all $\a$.

Since the map $z\To g(x_0,z)$ is a-upper semicontinuous at $y,$ there exists a net $v_\a\subseteq Z,\,v_\a\To v$, such that
$$g(x_0,y_\a)\in-C+v_\a\mbox{ and } v\in -C+g(x_0,y).$$

Consequently, $u_\a+g(x_0,y_\a)\in-C+v_\a+u_\a$ which combined with  $u_\a+g(x_0,y_\a)\not\in-\inte C$ for all $\a$ leads to
 $$u_\a+v_\a\not\in-\inte C\mbox{ for all }\a.$$

  Hence, $u_\a+v_\a\in Z\setminus-\inte C$ and $Z\setminus-\inte C$ is a closed set,  therefore by taking the limit we obtain that $u+v\in Z\setminus-\inte C,$ or, in other words $u+v\not\in-\inte C.$
On the other hand,
$u+v\in-C+f(x_0,y)-C+g(x_0,y)=-C+(f(x_0,y)+g(x_0,y))$
which shows that $$f(x_0,y)+g(x_0,y)\not\in-\inte C.$$

  Since  $y\neq x_0$ was arbitrary chosen the latter relation holds for every $y\in K\setminus\{x_0\}$.
\end{proof}

\begin{remark}\label{r1} If we assume that $f(x_0,x_0)+g(x_0,x_0)\not\in-\inte C$ in the hypothesis of Theorem \ref{principal}, then its conclusion holds for every $y\in K.$ Note that $f(x_0,x_0)+g(x_0,x_0)\not\in-\inte C$  if, for instance, $f(x_0,x_0)\in C$ and $g(x_0,x_0)\not\in-\inte C.$  Indeed, assume that $f(x_0,x_0)+g(x_0,x_0)\in-\inte C.$ Then, $g(x_0,x_0)\in-\inte C-f(x_0,x_0)\subseteq -\inte C-C=-\inte C,$ which is a contradiction.
\end{remark}

In what follows we provide conditions that assure the solution set of problem (\ref{p2}) is included in the solution set of the problem (\ref{p1}).
Assume that $x_0\in K$ is a solution of \eqref{p2}, that is $g(x_0,y)\not\in-\inte C$ for all $y\in K.$ Then the conditions (A$_1$)-(A$_5$) become the following.\\

(B$_1$) For every fixed $y\in K,\,y\neq x_0$ the map $x\To f(x,y)$ is a-upper semicontinuous at $x_0$ and there exists the nets $(x_\a),(z_\a)\subseteq K,\, \lim x_\a=x_0$   such that
\begin{equation}\label{cond1}g(x_0,z_\a)-f(x_\a,y)-g(x_0,y)\in- C,\mbox{ for all }\a.
\end{equation}

(B$_2$) For every fixed $y\in K,\,y\neq x_0$   the map $z\To f(x_0,z)$ is a-upper semicontinuous at $y$ and there exists the nets $(y_\a),(z_\a)\subseteq K,\, \lim y_\a=y$  such that \begin{equation}\label{cond2} g(x_0,z_\a)-f(x_0,y_\a)-g(x_0,y)\in- C,\mbox{ for all }\a.\end{equation}

(B$_3$) For every fixed $y\in K,\,y\neq x_0$  the maps $x\To f(x,y)$  and $x\To g(x,y)$ are  a-upper semicontinuous at $x_0$ and there exists the nets $(x_\a),(u_\a),(z_\a)\subseteq K,\, \lim x_\a=x_0,\,\lim u_\a=x_0$ such that \begin{equation}\label{cond3} g(x_0,z_\a)-f(x_\a,y)-g(u_\a,y)\in-  C,\mbox{ for all }\a.\end{equation}

(B$_4$) For every fixed $y\in K,\,y\neq x_0$ the maps $z\To f(x_0,z)$ and $z\To g(x_0,z)$ are a-upper semicontinuous at $y$ and there exist the nets $(y_\a),(v_\a),(z_\a)\subseteq K,$ $\lim y_\a=y,$ $\lim v_\a=y$  such that \begin{equation}\label{cond4} g(x_0,z_\a)-f(x_0,y_\a)-g(x_0,v_\a)\in- C,\mbox{ for all }\a.\end{equation}

(B$_5$) For every fixed $y\in K,\,y\neq x_0$ the map $x\To f(x,y)$ is a-upper semicontinuous at $x_0$, the map $z\To g(x_0,z)$ is a-upper semicontinuous at $y$ and there exist the nets $(x_\a),(y_\a),(z_\a)\subseteq K,\,\lim x_\a=x_0,\, \lim y_\a=y$  such that
\begin{equation}\label{cond5} g(x_0,z_\a)-f(x_\a,y)-g(x_0,y_\a)\in- C,\mbox{ for all }\a.\end{equation}

(B$_6$) For every fixed $y\in K,\,y\neq x_0$  the map $x\To g(x,y)$ is a-upper semicontinuous at $x_0$, the map $z\To f(x_0,z)$ is a-upper
semicontinuous at $y$ and there exist the nets $(x_\a),(y_\a),(z_\a)\subseteq K,\,\lim x_\a=x_0,\, \lim y_\a=y$  such that
\begin{equation}\label{cond6} g(x_0,z_\a)-f(x_0,y_\a)-g(x_\a,y)\in- C,\mbox{ for all }\a.\end{equation}

\begin{remark}
Note that we split (A$_5$) in two conditions (B$_5$) and (B$_6$). This is due to the fact that in (A$_5$) the role of $f$ and $g$ can be  interchanged, meanwhile in  (B$_5$) and (B$_6$) $g$ has the role of perturbation bifunction which  furnish the element $x_0$, therefore the semicontinuity conditions assumed at $x_0$ and $y$, respectively, give rise to different conditions.  By using Lemma \ref{notin}, the interested reader may easily rewrite \eqref{cond1}-\eqref{cond6}, taking into account that $g(x_0,z_\a)\not\in-\inte C$ for all $\a.$ An immediate consequence of Theorem \ref{principal} is the following.
\end{remark}
\begin{theorem}\label{t3.0}   Let $X$ be a Hausdorff space and let $Z$ be a Hausdorff  topological vector space,  let $C\subseteq Z$ be a convex and pointed cone with nonempty interior and let $K$ be a nonempty subset of $X$. Consider the mappings $f,g :K \times K \longrightarrow{Z}$. Let $x_0\in K$ such that $g(x_0,y)\not\in-\inte C,$ for all $y\in K$ and assume that $f(x_0,x_0)\in C.$
Assume further, that one of the statements (B$_1$)-(B$_6$)  hold.
Then $f(x_0,y)+g(x_0,y)\not\in-\inte C$ for all $y\in K.$
\end{theorem}
\begin{proof} By taking $w_\a=g(x_0\,z_\a)$ for all $\a$ in the hypothesis of Theorem \ref{principal} we get that $f(x_0,y)+g(x_0,y)\not\in-\inte C$ for all $y\in K,\,y\neq x_0.$

Remark \ref{r1} assures that the conclusion also holds for $y=x_0.$
\end{proof}

\begin{remark}\rm If one take a particular net, for instance the net $x_t=(1-t)x_0+t y,\,t\in]0,1[$ in assumptions (B$_1$)-(B$_6$), provided $K$ is a convex subset of the Hausdorff topological vector space $X$, then  Theorem \ref{t3.0} gives rise to some particular instances where  the a-upper semicontinuity  assumption can be weakened. We show this in case (B$_1$), the other cases are left to the interested reader. In this case, the assumption that the map $x\To f(x,y)$ is a-upper semicontinuous at $x_0$ for every $y\in K$ can be replaced by the following condition: if for some $y\in K$ one has $f((1-t)x_0+ty,y)+g(x_0,y)\not\in-\inte C$ for all $t\in]0,1],$ then $f(x_0,y)+g(x_0,y)\not\in-\inte C.$ Moreover, this condition is assured by the hemicontinuity of the mapping $x\To f(x,y)$  for all $y\in K$. Indeed, in this case
$\lim_{t\To 0}(f((1-t)x_0+ty,y)+g(x_0,y))=f(x_0,y)+g(x_0,y)$, hence $f((1-t)x_0+ty,y)+g(x_0,y)\in Z\setminus-\inte C$ implies $f(x_0,y)+g(x_0,y)\in Z\setminus-\inte C$ by taking the limit $t\To 0$ and taking into account that the set $Z\setminus-\inte C$  is closed.
\end{remark}

Consequently the following result holds.

\begin{corollary}\label{c13.0}   Let $X$ and $Z$ be  Hausdorff  topological vector spaces,  let $C\subseteq Z$ be a convex and pointed cone with nonempty interior and let $K$ be a nonempty, convex subset of $X$. Consider the mappings $f,g :K \times K \longrightarrow{Z}$. Let $x_0\in K$ such that $g(x_0,y)\not\in-\inte C,$ for all $y\in K.$
Assume that the following statements hold:
\begin{itemize}
\item[(i)] $f(x,x)\in C$ for all $x\in K,$
\item[(ii)] For every fixed $y\in K,\,y\neq x_0$ and for every  $t\in]0,1[$ there exists $z=z(t,y)\in  K$ such that
$g(x_0,z)-f((1-t)x_0+ty,y)-g(x_0,y)\in- C,$
\item[(iii)] If for some $y\in K$ one has $f((1-t)x_0+ty,y)+g(x_0,y)\not\in-\inte C$ for all $t\in]0,1],$ then $f(x_0,y)+g(x_0,y)\not\in-\inte C.$
\end{itemize}
Then $f(x_0,y)+g(x_0,y)\not\in-\inte C$ for all $y\in K.$
\end{corollary}
\begin{proof} Let $y\in K,\,y\neq x_0$ fixed. Obviously, $g(x_0,z)\not\in-\inte C$ for all $z\in K.$ Lemma \ref{notin} and (ii) assures that \begin{equation}\label{e1}f((1-t)x_0+ty,y)+g(x_0,y)\not\in-\inte C\mbox{ for all }t\in]0,1[.\end{equation}

Even more, (\ref{e1})  also holds for $t=1,$ that is  $f(y,y))+g(x_0,y)\not\in-\inte C.$ Indeed, assume the contrary that is $f(y,y))+g(x_0,y)\in-\inte C.$ From (i) one has $f(y,y)\in C$ hence $g(x_0,y)\in-\inte C-f(y,y))\subseteq -\inte C,$ which contradicts the fact that $g(x_0,y)\not\in-\inte C.$

 Hence, for all $t\in ]0,1]$ one has $f((1-t)x_0+ty,y)+g(x_0,y)\not\in-\inte C.$

 Now by   (iii) we have that $f(x_0,y)+g(x_0,y)\not\in-\inte C.$

 Since the latter relation holds also for $y=x_0$ and $y$ was arbitrary chosen the conclusion follows.
\end{proof}

\begin{remark}\rm
Note that assumption (ii) in the hypothesis of Corollary \ref{c13.0} holds, for instance, if for every $y\in K$ the function $\phi_y:[0,1]\To Z,$ $\phi_y(t)=g(x_0,(1-t)x_0+ty)-f((1-t)x_0+ty,y)$ has a maximum at $t=1$, i.e.,
$$\phi_y(1)-\phi_y(t)\in C,\mbox{ for all }t\in[0,1].$$
Indeed, $\phi_y(1)-\phi_y(t)=g(x_0,y)-f(y,y)-g(x_0,(1-t)x_0+ty)+f((1-t)x_0+ty,y)\in C$ is equivalent to $g(x_0,(1-t)x_0+ty)-f((1-t)x_0+ty,y)-g(x_0,y)\in -f(y,y)-C\subseteq -C$ and one can take $z=(1-t)x_0+ty$ in (ii).
\end{remark}
\begin{remark}
In what follows we show that the  assumptions \eqref{cond1}-\eqref{cond6} in the hypothesis of Theorem \ref{t3.0} are essential. We show this in case of (B$_1$). More precisely, we give an example such that all the assumptions in the hypothesis of Theorem \ref{t3.0} (with (B$_1$)) hold excepting \eqref{cond1}, and for an $x_0\in K$ one has $g(x_0,y)\not\in-\inte C$ for all $y\in K$, but there exists $y_0\in K$ such that $f(x_0,y)+g(x_0,y)\in-\inte C.$
\end{remark}

\begin{example}\rm [see also \cite{L1}, Example 3.2] \label{ex3.1} Let us consider the real valued functions of one real variable $\phi,\,\psi: K\To K,$ where $K=[-1,1]$ and \\ $\phi(x)=\left\{
\begin{array}{lll}
-2x-1,\,\mbox{if\,}\, x\in\left[-1,-\frac12\right],\\
2x+1,\,\mbox{if\,}\, x\in\left(-\frac12,0\right],\\
-2x+1,\,\mbox{if\,}\,x\in\left(0,1\right],\\
\end{array}
\right.$
$\psi(x)=\left\{
\begin{array}{ll}
-\frac23 x+\frac13,\,\mbox{if,}\, x\in\left[-1,\frac12\right],\\
-2x+1,\,\mbox{if,}\, x\in\left(\frac12,1\right].\\
\end{array}
\right.$
\end{example}
Consider the bifunctions $f,g:K\times K\To \R^2,$ $f(x,y)=\left(\phi(y)\psi(x),0\right),$  and $g(x,y)=\left(-\phi(x)\psi(x),\phi(y)\psi(x)-\phi(x)\psi(x)\right).$
Further, consider $C=\R_+^2=\{(x_1,x_2)\in\R^2:x_1\ge 0,\,x_2\ge0\}$ the nonnegative orthant of $\R^2$, which is obviously a convex and pointed cone, and
$\inte C= \{(x_1,x_2)\in\R^2:x_1> 0,\,x_2>0\}$. Hence, $(x,y)\in-\inte C$ iff $x<0$ and $y<0.$ We  consider the problem (\ref{p1}) and (\ref{p2}) defined by the bifunctions $f$ and $g$ and by the cone $C.$
Obviously the set $K=[-1,1]$ is convex and even compact, and $\phi(x)\psi(x)\ge 0$ for all $x\in K.$

  We show that  $x_0=-\frac12\in K$ is a solution of the (\ref{p2}). Indeed, $$g(x_0,y)=\left(-\phi(x_0)\psi(x_0),\phi(y)\psi(x_0)-\phi(x_0)\psi(x)\right)=\left(0, \frac23\phi(y)\right)\not\in-\inte C,\,\forall y\in K.$$

Further, $f(x,x)=(\phi(x)\psi(x),0)\in C$ for all $x\in K$.

Since the functions $\phi$ and $\psi$ are continuous, the map  $x\To f(x,y)$ is C-upper semicontiuous for  every fixed $y\in K$, hence, according to Remark \ref{ab}  the map  $x\To f(x,y)$ is  a-upper semicontinuous for  every fixed $y\in K$.

 We show that \eqref{cond1} in the hypothesis of Theorem \ref{t3.0} does not hold for $x_0=-\frac12$. Obviously, instead of nets we can consider sequences. Let $(x_n)\subseteq K,\, \lim_{n\To\infty}x_n=x_0.$ Let $y=\frac34$ and $(z_n)\subseteq K.$ Then,
  $$g(x_0,z_n)-f(x_n,y)-g(x_0,y)=\left(\frac12\psi(x_n),\left(\phi(z_n)+\frac12\right)\cdot\frac23\right).$$
  Since $\psi$ is continuous and $\psi(x_0)=\frac23>0$, there exists some $n_0\in\N$ such that $\psi(x_n)>0$ for all $n\ge n_0.$
  This shows that $$\left(\frac12\psi(x_n),\left(\phi(z_n)+\frac12\right)\cdot\frac23\right)\not\in- C,\,\forall (z_n)\subseteq K,\,\forall n\ge n_0,$$
  in other words \eqref{cond1} fails to hold.

 We show that  $x_0=-\frac12\in K$  is not a solution of (\ref{p1}).
Indeed,  for $y=\frac34\in K$ we obtain
$$f(x_0,y)+g(x_0,y)=\left(-\frac13,-\frac13\right)\in-\inte C.$$

\begin{remark} The next example shows that the a-upper semicontinuity assumptions in (B$_1$)-(B$_6$) in the hypothesis of Theorem \ref{t3.0} are essential. We show this in case of (B$_1$). More precisely, we give an example such that all the assumptions in the hypothesis of Theorem \ref{t3.0} (with (B$_1$)) hold excepting the map $x\To f(x,y)$ is a-upper semicontinuous at $x_0\in K$ for some $y\in K$, and one has $g(x_0,y)\not\in-\inte C$ for all $y\in K$, but there exists $y_0\in K$ such that $f(x_0,y_0)+g(x_0,y_0)\in-\inte C.$
\end{remark}

\begin{example} Consider the mappings $f,g: [-1,1]\times [-1,1] \To \R^2$ defined by\\

$ f(x,y) =\left\{\begin{array}{ll} \left(-1-x-y,x-y\right),\mbox{ if }x \in\left[-1,-\frac12\right],\\
\left(x+|y|,-x+y\right),\mbox{ if }x \in\left]-\frac12,1\right],
\end{array}\right.$ and $g(x,y)=(-1-x-|y|,|y|).$

Consider $C=\R^2_+$ the nonnegative orthant of $\R^2$ which is obviously a closed convex and pointed cone with nonempty interior. Further, the set $K=[-1,1]$ is convex and compact.

Then, for $x_0=-\frac12$ one has $g(x_0,y)\not\in-\inte C$ for all $y\in K$ and $f(x_0,x_0)=(0,0)\in C.$

Further, for any $y\in K,\,y\neq x_0$ fixed and for $(x_n)\subseteq K,\,x_n=-\frac{n}{2n+1}\To -\frac12$ and $(z_n)\subseteq K,\,z_n=\frac{n}{2n+1},$ one has
$g(x_0,z_n)-f(x_n,y)-g(x_0,y)=(0,-y-|y|)\in-C.$

We show, that for $y=0$ the map $x\To f(x,y)$ is not a-upper semicontinuous at $x_0.$

In other words, we show that for every sequence $(w_n)=(w_n^1,w_n^2)\subseteq \R^2,\,w_n\To (w^1,w^2)$ one has
$$f(x_n,0)-w_n\in -C\Longrightarrow f(x_0,0)-(w^1,w^2)\not\in C.$$

Indeed, let $(w_n)=(w_n^1,w_n^2)\subseteq \R^2,\,w_n\To (w^1,w^2)$ and assume that $$f(x_n,0)-w_n=\left(-\frac{n}{2n+1}-w_n^1,\frac{n}{2n+1}-w_n^2\right)\in -C.$$
Then $w^1\ge -\frac12$ and $w^2\ge \frac12.$

On the other hand $f(x_0,0)-(w^1,w^2)=\left(-\frac12-w^1,-\frac12-w^2\right)$ and since  $-\frac12-w^2\le-1$ we get that $f(x_0,0)-(w^1,w^2)\not\in C.$

Hence, all the assumptions in the hypothesis of Theorem \ref{t3.0} (with (B$_1$)) hold excepting the map $x\To f(x,y)$ is a-upper semicontinuous at $x_0$ for  $y=0.$ We show that there exists $y_0\in K$ such that $f(x_0,y_0)+g(x_0,y_0)\in-\inte C.$ Indeed, let $y_0=0.$ Then,
$$f(x_0,y_0)+g(x_0,y_0)=\left(-1,-\frac12\right)\in-\inte C.$$
\end{example}

\begin{remark} Note that Theorem \ref{t3.0} is a great theoretical result and assures solution existence of the perturbed problem (\ref{p1}) via some very simple conditions imposed on the bifunction $f.$ However, we need a priori to now a solution of (\ref{p2}), the weak vector equilibrium problem governed by the perturbation bifunction $g.$ In what follows  based on the a-upper semicontinuity notion we obtain the solution existence of problem \eqref{p1} and \eqref{p2}. We work in a noncompact setting, hence we need to use a coercivity  condition.
\end{remark}

Inspired from \cite{kazmi},  we will work with a coercivity condition concerning a compact set and its algebraic interior.
Let $X$  be a Hausdorff  topological vector space, let $U,V\subseteq X$ be convex sets and assume that $U\subseteq V$. We recall that the algebraic interior of $U$ relative to
$V$ is defined as
$$\mbox{core}_V U=\{u\in U: U\cap ]u,v]\neq\emptyset,\,\forall v\in V \}$$ Note that $\mbox{core}_V V = V.$
Our coercivity condition concerning the problem (\ref{p2}) becomes:

  There exists  $K_0\subseteq K$ nonempty, compact and convex such that for every $x\in K_0\setminus\mbox{core}_K K_0$ there exists an $y_0\in \mbox{core}_K K_0$ such that $g(x,y_0)\in -C.$

In the following results we use the coercivity conditions emphasized above and we  drop the  closedness condition on $K$ which is usually  assumed in  the equilibrium problems from the literature. Moreover, our conditions are assumed relative to this compact set $K_0$.

Another main  result of this section is the following.

\begin{theorem}\label{t37} Let $X$ and $Z$ be  Hausdorff  topological vector spaces,  let $C\subseteq Z$ be a convex and pointed cone with nonempty interior and let $K$ be a nonempty, convex subset of $X$. Consider the mappings $f,g :K \times K \longrightarrow{Z}$  and assume that the following
conditions are fulfilled.
\begin{itemize}
\item[(i)] There exists a nonempty compact convex subset $K_0$ of $K$ with the property that for every $x\in K_0\setminus\mbox{core}_K K_0,$ there exists an $y_0\in \mbox{core}_K K_0$ such that $g(x,y_0)\in -C.$
\item[(ii)]For all $y\in K_0,$ the mapping $x\To g(x,y)$ is a-usc on $K_0,$ the mapping $y\longrightarrow g(x,y)$ is C-convex on $K_0$ and $g(x,x)\not\in -\inte C$ for all $x\in K_0.$
\end{itemize}
Then, there exists $x_0\in K$ such that $g(x_0,y)\not\in-\inte C$ for all $y\in K.$
Moreover, assume that the following conditions hold.
\begin{itemize}
\item[(iii)]  One of the conditions (B$_1$)-(B$_6$)  is satisfied and $f(x_0,x_0)\in C$.
\end{itemize}

Then, $f(x_0,y)+g(x_0,y)\not\in -\inte{C}$ for all $y\in K.$
\end{theorem}
\begin{proof} We show at first that \eqref{p2} has a solution in $K_0$, that is, there exists $x_0\in K_0$ such that $g(x_0,y)\not\in-\inte C$ for all $y\in K_0.$

Assume the contrary, that is, for every $x\in K_0$ there exists $y\in K_0$ such that $g(x,y)\in-\inte C.$ Then, for every $y\in K_0$ consider $V_y=\{x\in K_0: g(x,y)\in-\inte C\}.$ It is obvious that $K_0\subseteq \cup_{y\in K_0} V_y.$ We show that $(V_y)_{y\in K_0}$ is an open cover of $K_0.$ First of all observe that for all $y\in K_0,$ one has  $V_y=K_0\setminus G(y)$, where $G(y)$ is the set $\{x\in K_0:g(x,y)\not\in-\inte C\}.$ But, by assumption (ii), for all $y\in K_0,$ the mapping $x\To g(x,y)$ is a-usc on $K_0,$ hence according to Lemma \ref{primclosed} $G(y)$ is closed for all $y\in K_0.$

Consequently, $(V_y)_{y\in K_0}$ is an open cover of the compact set $K_0$, in conclusion it contains a finite subcover. In other words, there exists $y_1,y_2,...,y_n\in K_0$ such that $K_0\subseteq \cup_{i=1}^n V_{y_i}.$ Consider $\big(p_i\big)_{i=\overline{1,n}}$  a continuous partition of  unity associated to the open cover $\big(V_{y_i}\big)_{i=\overline{1,n}}$. Then $p_i:K_0\To[0,1]$ is continuous and $\mbox{supp}(p_i)=\cl\{x\in K_0: p_i(x)\neq 0\}\subseteq V_{y_i}$ for all $i\in\{1,2,...,n\}$, moreover $\sum_{i=1}^n p_i(x)=1,$ for all $x\in K_0.$

Consider the mapping $\varphi:\co\{y_1,y_2,...,y_n\}\To\co\{y_1,y_2,...,y_n\},$
$$\varphi(x)=\sum_{i=1}^n p_i(x)y_i.$$
Obviously $\varphi$ is continuous, and $\co\{y_1,y_2,...,y_n\}$ is a compact and convex subset of the  finite dimensional  space $\mbox{span}\{y_1,y_2,...,y_n\}.$ Hence, by the Brouwer fixed point theorem, there exists $x_0\in \co\{y_1,y_2,...,y_n\}$ such that $\varphi(x_0)=x_0.$

Let $J=\{i\in\{1,2,...,n\}:p_i(x_0)>0\}.$ Obviously $J$ is nonempty, since $\sum_{i\in J} p_i(x_0)=1,$ and $$\varphi(x_0)=\sum_{i=1}^n p_i(x_0)y_i=\sum_{i\in J}p_i(x_0)y_i=x_0.$$
The latter  equality shows, that $x_0\in\co\{y_i:i\in J\}.$ On the other hand, from $p_i(x_0)>0$ for all $i\in J$ we obtain that $x_0\in\cap_{i\in J}V_{y_i}.$ Since $\cap_{i\in J}V_{y_i}$ is open, we  obtain  that there exists $U$ an open and convex neighbourhood of $x_0$ such that $U\subseteq \cap_{i\in J}V_{y_i}.$ Since $\co\{y_i:i\in J\}\cap U\neq\emptyset$, we have that there exists  $y_0=\sum_{i\in J}\l_i y_i\in \co\{y_i:i\in J\}\cap U$, where $\l_i\ge0$ for all $i\in J$ and $\sum_{i\in J}\l_i=1$. By (ii), in the hypothesis of the theorem, one gets $g(y_0,y_0)\not\in-\inte C$. On the other hand, by using the assumption that the map $y\longrightarrow g(x,y)$ is C-convex on $K_0$ we get
$$g(y_0,y_0) =g(y_0,\sum_{i\in J}\l_i y_i)\le\sum_{i\in J}\l_i g(y_0,y_i),$$
which shows that $\sum_{i\in J}\l_i g(y_0,y_i)-g(y_0,y_0)\in C.$ But, $y_0\in U$, thus $g(y_0,y_i)\in-\inte C,$ for all $i\in J.$ Hence $\sum_{i\in J}\l_i g(y_0,y_i)\in -\inte C,$ which leads to $$-g(y_0,y_0)\in C-\sum_{i\in J}\l_i g(y_0,y_i)\subseteq\inte C,$$ contradiction.
In conclusion there exists $x_0\in K_0$ such that $g(x_0,y)\not\in-\inte C$ for all $y\in K_0.$

We show next that $g(x_0,y)\not\in-\inte C$ for all $y\in K.$ First we show, that there exists $z_0\in \mbox{core}_K K_0$ such that $g(x_0,z_0)\in -C.$ Indeed, if $x_0\in \mbox{core}_K K_0$ then let $z_0=x_0$ and the conclusion follows from (i). Assume now, that $x_0\in K_0\setminus\mbox{core}_K K_0. $ Then, according to (i),  there exists $z_0\in \mbox{core}_K K_0$ such that $g(x_0,z_0)\in -C.$

Let $y\in K,\, y\neq z_0.$ Then, since $z_0\in \mbox{core}_K K_0$, there exists $\l\in]0,1]$ such that $\l z_0 +(1-\l)y\in K_0,$ consequently $g(x_0,\l z_0 +(1-\l)y)\not\in-\inte C.$ From (ii) we have
$$\l g(x_0,z_0)+(1-\l)g(x_0,y)-g(x_0,\l z_0 +(1-\l)y)\in C$$
or, equivalently
$$(1-\l)g(x_0,y)-g(x_0,\l z_0 +(1-\l)y)\in C-\l g(x_0,z_0)\subseteq C.$$
Assume that $g(x_0,y)\in-\inte C.$ Then, $$-g(x_0,\l z_0 +(1-\l)y)\in -(1-\l)g(x_0,y)+ C\subseteq \inte C,$$ in other words
$$g(x_0,\l z_0 +(1-\l)y)\in-\inte C,$$
contradiction. Hence,
$g(x_0,y)\not\in-\inte C,$ for all $y\in K.$

According  to Theorem \ref{t3.0}, (iii) assures that
$$f(x_0,y)+g(x_0,y)\not\in-\inte C,\mbox{ for all }y\in K.$$
\end{proof}

\begin{remark} Observe that we can assume our coercivity condition (i) relative to $f+g$ and then  the assumptions in (iii) can be put relative to the set $K_0$. More precisely the following result holds.
\end{remark}
\begin{theorem}\label{t39} Let $X$ and $Z$ be  Hausdorff  topological vector spaces,  let $C\subseteq Z$ be a convex and pointed cone with nonempty interior and let $K$ be a nonempty, convex subset of $X$. Consider the mappings $f,g :K \times K \longrightarrow{Z}$  and assume that the following
conditions are fulfilled.
\begin{itemize}
\item[(i)] There exists a nonempty compact convex subset $K_0$ of $K$ with the property that for every $x\in K_0\setminus\mbox{core}_K K_0,$ there exists an $y_0\in \mbox{core}_K K_0$ such that $f(x,y_0)+g(x,y_0)\in -C.$
\item[(ii)]For all $y\in K_0,$ the mapping $x\To g(x,y)$ is a-usc on $K_0,$ , for all $x\in K_0,$ the mapping $y\longrightarrow g(x,y)$ is C-convex on $K_0$ and $g(x,x)\not\in-\inte C$ for all $x\in K_0.$
\end{itemize}
Then, there exists $x_0\in K$ such that $g(x_0,y)\not\in-\inte C$ for all $y\in K.$
Moreover, assume that the following conditions hold.
\begin{itemize}
\item[(iii)] One of the conditions (B$_1$)-(B$_6$)  assumed for the restriction of $f$ and $g$ on $K_0\times K_0$ is satisfied and $f(x_0,x_0)\in C$.
\end{itemize}

Then, $f(x_0,y)+g(x_0,y)\not\in -\inte{C}$ for all $y\in K.$
\end{theorem}
\begin{proof} One can show that there exists $x_0\in K_0$ such that $g(x_0,y)\not\in-\inte C$ for all $y\in K_0$ as in the proof of Theorem \ref{t37}.

In virtue of Theorem \ref{t3.0}, (iii) assures that $f(x_0,y)+g(x_0,y)\not\in-\inte C$ for all $y\in K_0$. The rest of the proof is  analogous to the proof of Theorem \ref{t37}, by replacing $g$ with $f+g$.
\end{proof}

\end{document}